\documentclass{amsart}
\usepackage{latexsym}
\usepackage{amstext}
\usepackage{amsmath}
\usepackage{amssymb}
\usepackage{amsthm}
\usepackage{url}
\usepackage{enumerate}

\newtheorem{theorem}{Theorem}[section]
\newtheorem{lemma}[theorem]{Lemma}

\theoremstyle{definition}
\newtheorem{definition}[theorem]{Definition}

\newtheorem{example}[theorem]{Example}
\newtheorem{question}[theorem]{Question}

\def \kbar {\overline{k}}

\def \Qbar {\overline{\mathbb{Q}}}
\def \O {\mathcal{O}}
\DeclareMathOperator{\Sym}{Sym}
\DeclareMathOperator{\Gal}{Gal}
\DeclareMathOperator{\Supp}{Supp}

\DeclareMathOperator{\Jac}{Jac}
\DeclareMathOperator{\rk}{rank}
\DeclareMathOperator{\dv}{div}

\DeclareMathOperator{\Div}{Div}
\DeclareMathOperator{\codim}{codim}

\begin{document}
\bibliographystyle{amsplain}
\title{Integral points of bounded degree on affine curves}
\author{Aaron Levin}
\address{Department of Mathematics\\Michigan State University\\East Lansing, MI 48824}
\curraddr{}
\email{adlevin@math.msu.edu}
\begin{abstract}
We generalize Siegel's theorem on integral points on affine curves to integral points of bounded degree, giving a complete characterization of affine curves with infinitely many integral points of degree $d$ or less over some number field.  Generalizing Picard's theorem, we prove an analogous result characterizing complex affine curves admitting a nonconstant holomorphic map from a degree $d$ (or less) analytic cover of $\mathbb{C}$.
\end{abstract}

\maketitle

\section{Introduction}

Let $C\subset \mathbb{A}^n$ be a nonsingular affine curve defined over a number field $k$ and let $\tilde{C}$ be a nonsingular projective closure of $C$.  Let $S$ be a finite set of places of $k$ containing the archimedean places.  Siegel's classic theorem \cite{Sie} states that the set of $S$-integral points $C(\O_{k,S})=C(k)\cap \mathbb{A}^n(\O_{k,S})$ is finite if either $\tilde{C}$ has positive genus or $C$ has more than two points at infinity (i.e., $\#(\tilde{C}\setminus C)(\kbar)>2$).  Conversely, if $\tilde{C}$ has genus zero and $C$ has two or fewer points at infinity, then for some finite extension $L$ of $k$ and some finite set of places $S$ of $L$, the set of $S$-integral points $C(\O_{L,S})$ will be infinite.  In this case, $C$ is isomorphic, over $\kbar$, to either $\mathbb{A}^1$ or $\mathbb{G}_m=\mathbb{P}^1\setminus \{0,\infty\}=\mathbb{A}^1\setminus \{0\}$.  Since $\mathbb{G}_m\subset \mathbb{A}^1$, we can state Siegel's theorem as follows.

\begin{theorem}[Siegel's theorem]
\label{Sieg}
Let $C\subset \mathbb{A}^n$ be a nonsingular affine curve defined over a number field $k$.  Then there exists a finite extension $L$ of $k$ and a finite set of places $S$ of $L$ such that the set $C(\O_{L,S})$ is infinite if and only if there exists $C'\subset C$ with $C'$ isomorphic, over $\kbar$, to $\mathbb{G}_m$.
\end{theorem}

In fact, one can give precise necessary and sufficient conditions for the infinitude of $C(\O_{k,S})$ (see \cite{ABP}).  We will generalize Siegel's theorem to integral points of bounded degree.  There is one obvious way to construct infinitely many integral points of degree $d$ or less on an affine curve: one may pull back $S$-integral points on $\mathbb{G}_m$ (i.e., $S$-units) via an appropriate map of degree $d$ or less.  Our main result asserts that this obvious construction essentially explains all affine curves with an infinite set of integral points of degree $d$ or less.  Let $\overline{\O}_{k,S}$ denote the integral closure of $\O_{k,S}$ in $\kbar$.  If $L$ is a finite extension of  a number field $k$ and $S$ is a set of places of $k$, then let $S_L$ denotes the set of places of $L$ lying above places in $S$.

\begin{theorem}
\label{mtheorem2}
Let $C\subset \mathbb{A}^n$ be a nonsingular affine curve defined over a number field $k$.  Let $d$ be a positive integer.  Then there exists a finite extension $L$ of $k$ and a finite set of places $S$ of $L$ such that the set of $S$-integral points of degree $d$ or less over $L$,
\begin{align}
\label{infset}
\{P\in C(\overline{\O}_{L,S})\mid [L(P):L]\leq d\}=\bigcup_{\substack{M\supset L\\ [M:L]\leq d}}C(\O_{M,S_M}),
\end{align}
is infinite if and only if there exists, over $\kbar$, an affine curve $C'\subset C$ and a finite morphism $\phi:C'\to \mathbb{G}_m$ with $\deg \phi\leq d$.
\end{theorem}

If $C\subset \mathbb{A}^n$ is a singular affine curve, then there exists an infinite set \eqref{infset} (for some $L$ and $S$) if and only if the same is true of the normalization of $C$ (with some affine embedding).  Thus, there is no loss of generality in considering nonsingular affine curves.  Theorem \ref{mtheorem2} may alternatively be stated in terms of points at infinity as follows.

\begin{theorem}
\label{mtheorem}
Let $C\subset \mathbb{A}^n$ be a nonsingular affine curve defined over a number field $k$.  Let $\tilde{C}$ be a nonsingular projective completion of $C$ and let $(\tilde{C}\setminus C)(\kbar)=\{P_1,\ldots, P_q\}$.  Let $d$ be a positive integer.  Then there exists a finite extension $L$ of $k$ and a finite set of places $S$ of $L$ such that the set
\begin{align*}
\{P\in C(\overline{\O}_{L,S})\mid [L(P):L]\leq d\}
\end{align*}
is infinite if and only if there exists a morphism $\phi:\tilde{C}\to\mathbb{P}^1$, over $\kbar$, with $\deg \phi\leq d$ and $\phi(\{P_1,\ldots, P_q\})\subset \{0,\infty\}$.
\end{theorem}

We will prove Theorem \ref{mtheorem2} in this form (see Lemma \ref{lequiv} for the equivalence of the two statements).
Using the Schmidt Subspace Theorem, Corvaja and Zannier \cite{CZ} proved the case $d=2$ and $q\geq 4$ of Theorem \ref{mtheorem} (in fact, in this case they proved the stronger statement of Theorem \ref{mtheorem3}\eqref{mii}).  The case $d=2$ and $q=3$ was proved for nonhyperelliptic curves in \cite[Th.\ 6]{Lev}.  If $q>2d$, then there cannot exist a morphism $\phi$ as in Theorem \ref{mtheorem} and so the set $\{P\in C(\overline{\O}_{L,S})\mid [L(P):L]\leq d\}$ is always finite.  This consequence of Theorem \ref{mtheorem} was noted previously \cite[Cor.~14.14]{Lev4} as a consequence of a Diophantine approximation inequality of Vojta \cite{V4}.  Finding necessary and sufficient conditions similar to Theorems \ref{mtheorem2} and \ref{mtheorem} for a fixed ground field $k$ and finite set of places $S$ of $k$ appears to be a subtle issue (see Examples \ref{ex1} and \ref{ex2}).

As first noticed by Osgood and Vojta, there is a surprising correspondence between certain statements in Diophantine approximation and certain statements in Nevanlinna theory.  We refer the reader to \cite{V2} for Vojta's dictionary between the two subjects.  At the qualitative level, Siegel's theorem on integral points on curves is related to Picard's well-known theorem that a nonconstant holomorphic function omits at most one finite complex value.  More precisely, Siegel's theorem corresponds to the following result of Picard, which we state in a way analogous to Theorem \ref{Sieg}.
\begin{theorem}
\label{tPic}
Let $C$ be a nonsingular complex affine curve.  Then there exists a nonconstant holomorphic map $f:\mathbb{C}\to C$ if and only if there exists $C'\subset C$ with $C'$ isomorphic to $\mathbb{G}_m$.
\end{theorem}

Unsurprisingly, using Vojta's dictionary, the proof of Theorem \ref{mtheorem} can be used to obtain an analogous result in value distribution theory, generalizing Picard's theorem.  An infinite set of integral points of degree $d$ on a variety $X$ corresponds to a nonconstant holomorphic map from a $d$-sheeted covering of $\mathbb{C}$ to $X$.  We obtain the following theorem.
\begin{theorem}
\label{mtheoremb}
Let $\tilde{C}$ be a nonsingular complex projective curve.  Let $C\subset \tilde{C}$ be a complex affine curve and let $\tilde{C}\setminus C=\{P_1,\ldots, P_q\}$.  Let $d$ be a positive integer.  Then there exists a finite analytic covering $\pi:M\to \mathbb{C}$ with $\deg \pi\leq d$ and a nonconstant holomorphic map $f:M\to C$ if and only if there exists a morphism $\phi:\tilde{C}\to\mathbb{P}^1$ with $\deg \phi\leq d$ and $\phi(\{P_1,\ldots, P_q\})\subset \{0,\infty\}$.
\end{theorem}

For $M$ and $\tilde{C}$ as in Theorem \ref{mtheoremb}, define a point $P\in \tilde{C}$ to be a Picard exceptional point of a holomorphic map $f:M\to \tilde{C}$ if $P\not\in f(M)$.  Then we may view Theorem~\ref{mtheoremb} as giving a characterization of Picard exceptional points.

\begin{theorem}
\label{mtheoremb2}
Let $\tilde{C}$ be a nonsingular complex projective curve.  Let $\{P_1,\ldots, P_q\}\subset \tilde{C}$ be a nonempty set and let $d$ be a positive integer.  Then there exists a finite analytic covering $\pi:M\to \mathbb{C}$ with $\deg \pi\leq d$ and a nonconstant holomorphic map $f:M\to \tilde{C}$ with $P_1,\ldots, P_q$ Picard exceptional points of $f$ if and only if there exists a morphism $\phi:\tilde{C}\to\mathbb{P}^1$ with $\deg \phi\leq d$ and $\phi(\{P_1,\ldots, P_q\})\subset \{0,\infty\}$.
\end{theorem}

We are not able to determine the existence of surjective maps $f:M\to\tilde{C}$ (when there are no Picard exceptional points).  This problem is analogous to classifying curves with infinitely many algebraic points of degree $d$ over a number field.  As discussed at the end of this section, this appears to be a more subtle and difficult problem.

To prove Theorem \ref{mtheorem}, we study certain $k$-rational integral point sets on the $d$th symmetric power $\Sym^d(C)$ of a projective curve $C$.  The proof then involves an interplay between the Abel-Jacobi map $\Sym^d(C)\to\Jac(C)$ and deep results on integral points on subvarieties of abelian and semiabelian varieties (the unit equation and theorems of Faltings and Vojta).  

The proof of Theorem \ref{mtheoremb} is completely analogous to the proof of Theorem \ref{mtheorem}, and so we omit the proof.  Indeed, the proof of Theorem \ref{mtheorem} consists of a purely geometric argument combined with the deep arithmetic theorems of Section \ref{ssemi}.  All of the Diophantine results in Section~\ref{ssemi} have exact analogues in value distribution theory, as we now recall.  The analogue of Faltings' Theorem \ref{Falt1} for holomorphic curves in subvarieties of abelian varieties is known as Bloch's conjecture.  After work of Bloch \cite{Bloch}, substantial progress towards the conjecture was made by Ochiai \cite{Och} and a complete proof of the conjecture was given by Kawamata \cite{Kaw} (see also Noguchi-Ochiai \cite{NO}).  The analogous result for holomorphic curves in subvarieties of semiabelian varieties is due to Noguchi \cite{Nog}.  The hyperbolicity of the complement of an ample effective divisor on an abelian variety (the analogue of Faltings' Theorem~\ref{Fal2}) was proved by Siu and Yeung \cite{SY}.  This was again generalized to semiabelian varieties by Noguchi \cite{Nog2}.  Finally, we note that the analogue of Theorem \ref{sunit} on the unit equation is the Borel lemma.  In correspondence with the arithmetic case, one associates to the holomorphic map $f:M\to C$ a holomorphic map $g:\mathbb{C}\to \Sym^d(C)$, the so-called algebroid reduction of $f$ (see \cite{Stoll}).  The proof of Theorem \ref{mtheoremb} now proceeds in the exact same way as the proof of Theorem \ref{mtheorem}, substituting at appropriate steps the appropriate value distribution theory results in place of the arithmetic results.

After stating the basic definitions, in Section \ref{ssemi} we recall the key results on integral points on subvarieties of semiabelian varieties that underpin the proof of Theorem \ref{mtheorem}.  In Section \ref{smain} we give the proof of the main result.  In Section \ref{sopen} we give some examples clarifying our results and mention some possible strengthenings of Theorem \ref{mtheorem}.

Finally, for completeness, we recall the situation for algebraic points of bounded degree on curves (that is, without any integrality assumptions).  As detailed below, the natural analogue of Theorem \ref{mtheorem2} in this setting holds only when $d\leq 3$.  In view of this, it is perhaps a bit surprising that imposing an integrality condition is sufficient to ameliorate the situation.

Let $C$ be a nonsingular projective curve defined over a number field $k$.  Faltings' theorem (the Mordell conjecture) asserts that $C(L)$ is infinite for some finite extension $L$ of $k$ if and only if the genus of $C$ is zero or one.  If $C$ admits a degree $d$ morphism to the projective line or an elliptic curve, then by pulling back $k$-rational points via this morphism one sees that, after possibly replacing $k$ by a larger number field, the set
\begin{align*}
\{P\in C(\kbar)\mid [k(P):k]\leq d\}
\end{align*}
is infinite.  Harris and Silverman \cite{HS} proved the converse in the case $d=2$.

\begin{theorem}[Harris, Silverman]
Let $C$ be a nonsingular projective curve defined over a number field $k$.  Then the set 
\begin{align*}
\{P\in C(\kbar)\mid [L(P):L]\leq 2\}
\end{align*}
is infinite for some finite extension $L$ of $k$ if and only if $C$ is hyperelliptic or bielliptic.
\end{theorem}

More generally, we have the following theorem of Abramovich and Harris \cite{AH}.

\begin{theorem}[Abramovich, Harris]
\label{AHthm}
Let $d\leq 4$ be a positive integer.  Let $C$ be a nonsingular projective curve over a number field $k$ with genus not equal to $7$ if $d=4$.  Then the set
\begin{align*}
\{P\in C(\kbar)\mid [L(P):L]\leq d\}
\end{align*}
is infinite for some finite extension $L$ of $k$ if and only if $C$ admits a map of degree $\leq d$, over $\kbar$, to $\mathbb{P}^1$ or an elliptic curve.
\end{theorem}

Given Theorem \ref{AHthm}, Abramovich and Harris naturally conjectured that a similar result would hold for all $d$.  However, Debarre and Fahlaoui \cite{DF} gave counterexamples to the conjecture for all $d\geq 4$.  Frey \cite{Frey} has shown, however, that the existence of infinitely many algebraic points of degree $d$ on a curve implies that there exists a map of degree $\leq 2d$ to $\mathbb{P}^1$.

\begin{theorem}[Frey]
Let $d$ be a positive integer.  Let $C$ be a nonsingular projective curve defined over a number field $k$.  If the set
\begin{align*}
\{P\in C(\kbar)\mid [k(P):k]\leq d\}
\end{align*}
is infinite, then $C$ admits a map over $k$ of degree $\leq 2d$ to $\mathbb{P}^1$.
\end{theorem}

\section{Definitions}

Let $k$ be a number field.  We let $\mathcal{O}_k$ denote the ring of integers of $k$.  We have a canonical set $M_k$ of places of $k$ consisting of one place for each prime ideal $\mathfrak{p}$ of $\mathcal{O}_k$, one place for each real embedding $\sigma:k \to \mathbb{R}$, and one place for each pair of conjugate embeddings $\sigma,\overline{\sigma}:k \to \mathbb{C}$.  Let $k_v$ denote the completion of $k$ with respect to $v$.    For a finite set of places $S$ of $k$ containing the archimedean places, we let $\O_{k,S}$ denote the ring of $S$-integers of $k$, $\overline{\O}_{k,S}$ denote the integral closure of $\O_{k,S}$ in $\kbar$, and $\O_{k,S}^*$ denote the group of $S$-units of $k$.

There are various, essentially equivalent, notions of integral points on varieties.  It will be most convenient for us to use the language of $(D,S)$-integral points, which we now recall.  We will also briefly discuss the relation with other notions of integral points that we will occasionally use.

Let $D$ be a Cartier divisor on a projective variety $X$, with both $X$ and $D$ defined over a number field $k$.  We let $\Supp D$ denote the support of $D$.  From the theory of heights, for each place $v\in M_k$ we can associate to $D$ a local Weil function (or local height function) $\lambda_{D,v}:X(\kbar_v)\setminus \Supp D \to \mathbb{R}$, unique up to a bounded function (extending $|\cdot|_v$ to an absolute value on $\kbar_v$ in some way).  When $D$ is effective, the Weil function $\lambda_{D,v}$ gives a measure of the $v$-adic distance of a point to $D$, being large when the point is close to $D$.  By choosing embeddings $k\to \kbar_v$ and $\kbar\to \kbar_v$, we may also think of $\lambda_{D,v}$ as a function on $X(k)\setminus \Supp D$ or $X(\kbar)\setminus \Supp D$.  A global Weil function consists of a collection of local Weil functions, $\lambda_{D,v}$, $v\in M_k$, satisfying certain reasonable conditions as $v$ varies.  We refer the reader to \cite{BG, HinS, Lan, V2} for details on Weil functions.

\begin{definition}
\label{intdef}
Let $D$ be an effective Cartier divisor on a projective variety $X$, with both $X$ and $D$ defined over a number field $k$.  Let $S$ be a finite set of places of $k$ containing the archimedean places.  Let $R\subset X(\kbar)\setminus \Supp D$.  Then $R$ is called a set of $(D,S)$-integral points on $X$ if there exist constants $c_v$, $v\in M_k$, with $c_v=0$ for all but finitely many $v$, and a global Weil function $\lambda_{D,v}$ such that for all $P\in R$, all $v\in M_k\setminus S$, and all embeddings of $\kbar$ in $\kbar_v$,
\begin{align*}
\lambda_{D,v}(P)\leq c_v.
\end{align*}
\end{definition}

When $D$ is very ample, $(D,S)$-integral points coincide with the natural notion of $S$-integral points coming from affine models of $X\setminus \Supp D$ (see \cite[Lemma 1.4.1]{V2}).

\begin{theorem}
Let $D$ be a very ample effective divisor on a projective variety $X$, with both $X$ and $D$ defined over a number field $k$.  Let $R\subset X(\kbar)\setminus \Supp D$.  Then $R$ is a set of $(D,S)$-integral points on $X$ if and only if there exists an affine embedding $X\setminus \Supp D\hookrightarrow \mathbb{A}^n$ such that $R\subset (X\setminus \Supp D)(\kbar)\cap \mathbb{A}^n(\overline{\O}_{k,S})$.
\end{theorem}

There is also the scheme-theoretic notion of $S$-integral points which is used, for instance, in the statement of Theorem \ref{Voj1}.  Again, there is a natural relationship between sets of $(D,S)$-integral points on $X$ and $S$-integral points on models of $X\setminus \Supp D$ over $\O_{k,S}$.  Since this will not be directly used in our proofs, we refer the reader to \cite[Prop.\ 1.4.7]{V2} for the precise relationship.

We define the irregularity of a projective variety $X$ over a field of characteristic $0$ to be $q(X)=\dim H^1(X',\mathcal{O}_{X'})=\dim H^0(X',\Omega^1_{X'})$, where $X'$ is a desingularization of $X$.  This is independent of the choice of $X'$ and is equal to the dimension of the Albanese variety of $X'$.  Suppose now that $X$ is nonsingular.  If $D$ is a divisor on $X$, we let $c_1(D)$ denote the first Chern class of $D$.  For divisors $D$ and $E$ on $X$, we write $D\leq E$ if $E-D$ is an effective divisor.  We say that the effective divisors $D_1,\ldots, D_q$ on $X$ are in general position if for any subset $I\subset \{1,\ldots, q\}$, $|I|\leq \dim X+1$, we have $\codim \cap_{i\in I}\Supp D_i\geq |I|$.  If $G$ is an abelian group and $g_i$, $i\in I$, are elements of $G$, we let $\rk \{g_i\}_{i\in I}$ denote the (free) rank of the subgroup of $G$ generated by the elements $g_i,i\in I$. 

\section{Integral points on abelian and semiabelian varieties}
\label{ssemi}
In this section we recall the fundamental results on integral points on abelian and semiabelian varieties that will be used in the proof of our main theorem.  First, we recall Faltings' result \cite{Fal1,Fal2} on the structure of rational points on subvarieties of abelian varieties.

\begin{theorem}[Faltings]
\label{Falt1}
Let $X$ be a closed subvariety of an abelian variety $A$, with both $X$ and $A$ defined over a number field $k$.  Then the set $X(k)$ equals a finite union $\cup B_i(k)$, where each $B_i$ is a translated abelian subvariety of $A$ contained in $X$.
\end{theorem}

Vojta \cite{V1} generalized Faltings' theorem to cover integral points on closed subvarieties of semiabelian varieties.

\begin{theorem}[Vojta]
\label{Voj1}
Let $X$ be a closed subvariety of a semiabelian variety $A$, with both $X$ and $A$ defined over a number field $k$.  Let $S$ be a finite set of places of $k$ containing the archimedean places.  Let $\mathcal{X}$ be a model for $X$ over $\O_{k,S}$.  Then the set $\mathcal{X}(\O_{k,S})$ equals a finite union $\cup \mathcal{B}_i(\O_{k,S})$, where each $\mathcal{B}_i$ is a subscheme of $\mathcal{X}$ whose generic fiber $B_i$ is a translated semiabelian subvariety of $A$.
\end{theorem}

As noted by Vojta \cite[Cor.\ 0.3]{V1}, if an effective divisor $D$ on a projective variety $X$ has enough irreducible components, relative to certain geometric invariants of $X$, then $X\setminus \Supp D$ embeds into a semiabelian variety and one can use Theorem~\ref{Voj1} to show that any set of $k$-rational $(D,S)$-integral points on $X$ is not Zariski dense.  We will use this result in the following form (see \cite[Lemma 4.4]{NW}\footnote{This lemma is incorrectly stated in \cite{NW} with a term $\#\{D_i|_W\}_{i=1}^q$ instead of $\rk\{D_i|_W\}_{i=1}^q$.  The same proof in \cite{NW} gives, however, the statement in Lemma \ref{salem}.  An equivalent, but slightly different formulation, is given in \cite[Th.~ 4.9.7, Th.~9.7.5]{NW2}.} and \cite[Th.~4.9.7, Th.~9.7.5]{NW2}).
\begin{lemma}
\label{salem}
Let $D_1,\ldots, D_q$ be effective divisors on a nonsingular projective variety $V$.  Let $W\subset V$, $W\not\subset \Supp D_i$ for all $i$, be a subvariety of $V$ such that there exists a Zariski dense set of $k$-rational $(\sum_{i=1}^q D_i|_W,S)$-integral points on $W$.  Then
\begin{align*}
\rk\{D_i|_W\}_{i=1}^q+q(W)\leq \dim W +\rk\{c_1(D_i)\}_{i=1}^q.
\end{align*}
\end{lemma}

Note that $\rk\{D_i|_W\}_{i=1}^q$ is the rank of the group generated by $D_i|_W$, $i=1,\ldots, q$, as a subgroup of $\Div(W)$, the group of Cartier divisors on $W$ {\em with no equivalence relation imposed}.

We will also need a result on integral points on the complement of an effective divisor on an abelian variety.  In the case when the divisor is ample, Faltings \cite{Fal1} proved the finiteness of integral points on the complement.

\begin{theorem}[Faltings]
\label{Fal2}
Let $A$ be an abelian variety defined over a number field $k$.  Let $S$ be a finite set of places of $k$ containing the archimedean places and let $D$ be an ample effective divisor on $A$.  Then any set of $k$-rational $(D,S)$-integral points on $A$ is finite.
\end{theorem}

If the divisor is not ample, then we only have a degeneracy statement.

\begin{theorem}
\label{Voj2}
Let $A$ be an abelian variety defined over a number field $k$.  Let $S$ be a finite set of places of $k$ containing the archimedean places and let $D$ be a nontrivial effective divisor on $A$.  Then any set of $k$-rational $(D,S)$-integral points on $A$ is not Zariski dense in $A$.
\end{theorem}

More generally, Vojta \cite{V3} proved a generalization of Theorem \ref{Voj2} to semiabelian varieties.

Finally, we recall the fundamental result on the unit equation, proved independently by Evertse \cite{Ev} and van der Poorten and Schlickewei \cite{vdP2}.
\begin{theorem}[Evertse, van der Poorten and Schlickewei]
\label{sunit}
All but finitely many solutions of the equation
\begin{equation*}
u_1+u_2+\ldots+u_n=1, \quad u_1,\ldots,u_n\in \O_{k,S}^*,
\end{equation*}
satisfy an equation of the form $\sum_{i\in I}u_i=0$, where $I$ is a nonempty subset of $\{1,\ldots,n\}$.
\end{theorem}

Geometrically, this may be viewed as a result about integral points on the complement of a hyperplane in $\mathbb{G}_m^n$.  Thus, Theorem \ref{sunit} is again a result about integral points on a semiabelian variety.  Using Theorem \ref{sunit}, it is easy to classify when a complement of hyperplanes in projective space may contain a Zariski dense set of rational integral points.  We have the following result from \cite[Theorem 6A]{Lev3}.

\begin{lemma}
\label{hyplem}
Let $\mathcal{H}$ be a set of hyperplanes in $\mathbb{P}^n$ defined over a number field $k$ and let $\mathcal{L}$ be a corresponding set of linear forms.  There does not exist a Zariski dense set of $K$-rational $(\sum_{H\in\mathcal{H}}H,S)$-integral points on $\mathbb{P}^n$, for all number fields $K\supset k$ and $S\subset M_K$, if and only if $\mathcal{L}$ is a linearly dependent set.  Furthermore, in this case any set $R$ of $K$-rational $(\sum_{H\in\mathcal{H}}H,S)$-integral points on $\mathbb{P}^n$ is contained in a finite union of hyperplanes of $\mathbb{P}^n$ defined over $K$. 
\end{lemma}

\section{Proof of the main theorem}
\label{smain}



We first prove the equivalence of Theorems \ref{mtheorem2} and \ref{mtheorem}.  This is an immediate consequence of the following result.
\begin{theorem}
\label{lequiv}
Let $C$ be a nonsingular affine curve over an algebraically closed field $k$.  Let $\tilde{C}$ be a nonsingular projective closure of $C$ and let $(\tilde{C}\setminus C)(k)=\{P_1,\ldots, P_q\}$.  There exists an affine curve $C'\subset C$ and a finite morphism $\phi:C'\to \mathbb{G}_m$ of degree $d$ if and only if there exists a morphism $\phi:\tilde{C}\to\mathbb{P}^1$ of degree $d$ with $\phi(\{P_1,\ldots, P_q\})\subset \{0,\infty\}$.
\end{theorem}

\begin{proof}
Suppose there exists a morphism $\phi:\tilde{C}\to\mathbb{P}^1$ of degree $d$ with $\phi(\{P_1,\ldots, P_q\})\subset \{0,\infty\}$.  Let $C'=\phi^{-1}(\mathbb{G}_m)$.  As is well known, since $\phi$ is a nonconstant morphism between projective curves, $\phi$ is a finite morphism.  Then for any open subset $U\subset\mathbb{P}^1$, $\phi|_{\phi^{-1}(U)}:\phi^{-1}(U)\to U$ is a finite morphism.  In particular, $\phi|_{C'}:C'\to\mathbb{G}_m$ is a finite morphism of degree $d$ and by assumption $C'\subset C$.

Conversely, suppose that there exists an affine curve $C'\subset C$ and a finite morphism $\phi:C'\to \mathbb{G}_m$ of degree $d$.  We may extend $\phi$ to a morphism $\phi:\tilde{C}\to \mathbb{P}^1$.  Let $C''=\phi^{-1}(\mathbb{G}_m)$.  As noted earlier, $\phi:C''\to\mathbb{G}_m$ is a finite morphism.  Clearly $C'\subset C''$.  Since $C''\to \mathbb{G}_m$ and $C'\to\mathbb{G}_m$ are both finite morphisms, it follows that $C'\hookrightarrow C''$ is a finite morphism.  Since finite morphisms are closed, this implies that $C'=C''$.  It follows that $\phi(\{P_1,\ldots, P_q\})\subset \{0,\infty\}$.
\end{proof}

The easy direction of Theorems \ref{mtheorem2} and \ref{mtheorem} follows from the next result.  We construct infinitely many integral points of degree $d$ (or less) on $C$ by pulling back integral points on $\mathbb{G}_m$ (i.e., $S$-units) via the morphism $\phi$.

\begin{theorem}
Let $C\subset \mathbb{A}^n$ be a nonsingular affine curve and suppose that there exists a finite morphism $\phi:C\to \mathbb{G}_m$, with both $C$ and $\phi$ defined over a number field $k$.  Let $d=\deg\phi$.  Then for some finite set of places $S$ of $k$, the set of $S$-integral points of degree $d$ or less over $k$,
\begin{align*}
\{P\in C(\overline{\O}_{k,S})\mid [k(P):k]\leq d\},
\end{align*}
is infinite.
\end{theorem}
\begin{proof}
Let $x_1,\ldots, x_n$ be the $n$ coordinate functions on $\mathbb{A}^n$ restricted to $C$.  Since $\phi$ is a finite morphism, each $x_i$ satisfies a monic polynomial equation over $k[\phi,\frac{1}{\phi}]$.  Taking $S$ large enough so that $\O_{k,S}^*$ is infinite and all of the above monic polynomials have coefficients in $\O_{k,S}[\phi,\frac{1}{\phi}]$, we see that 
\begin{align*}
\phi^{-1}(\O_{k,S}^*)\subset \{P\in C(\overline{\O}_{k,S})\mid [k(P):k]\leq d\},
\end{align*}
and both sets are infinite.

Alternatively, consider a nonsingular projective completion $\tilde{C}$ of $C$ and extend $\phi$ to a morphism $\phi:\tilde{C}\to\mathbb{P}^1$.  Let $(\tilde{C}\setminus C)(\kbar)=\{P_1,\ldots, P_q\}$.  Let $D=\sum_{i=1}^qP_i$.  From the proof of Theorem \ref{lequiv}, $\phi(\{P_1,\ldots, P_q\})\subset \{0,\infty\}$.  Let $S'$ be a finite set of places of $k$ containing the archimedean places.  If $R$ is a set of $(0+\infty,S')$-integral points on $\mathbb{P}^1$, then from functoriality of Weil functions, it is immediate that $\phi^{-1}(R)$ is a set of $(\phi^*(0+\infty),S')$-integral points on $\tilde{C}$.  Since $D\leq \phi^*(0+\infty)$, $\phi^{-1}(R)$ is also a set of $(D,S')$-integral points on $\tilde{C}$.  Then for some finite set of places $S\supset S'$, depending on the embedding $C\subset \mathbb{A}^n$, we have $\phi^{-1}(R)\subset \{P\in C(\overline{\O}_{k,S})\mid [k(P):k]\leq d\}$.  Since $\O_{k,S'}^*$ is a set of $(0+\infty,S')$-integral points on $\mathbb{P}^1$, the result follows (for an appropriate $S$) on taking $R=\O_{k,S'}^*$ with $|S'|>1$ (so that $R$ is infinite).
\end{proof}

The difficult implication in Theorem \ref{mtheorem} follows from the next theorem.  In fact, when the affine curve is rational or has enough points at infinity, we obtain an even stronger statement.

\renewcommand{\theenumi}{\alph{enumi}}

\begin{theorem}
\label{mtheorem3}
Let $C$ be a nonsingular projective curve defined over a number field $k$.  Let $S$ be a finite set of places of $k$ containing the archimedean places.  Let $D=\sum_{i=1}^q P_i$ be a nontrivial sum of distinct rational points $P_i\in C(k)$.  Let $d$ be a positive integer.  Let
\begin{align*}
\Phi(D,d,k)&=\{\phi\in k(C)\mid \deg \phi\leq d, D\leq \phi^*(0)+\phi^*(\infty)\},\\
Z(D,d,k)&=\bigcup_{\phi\in \Phi(D,d,k)} \phi^{-1}(k).
\end{align*}
Suppose that there exists an infinite set of $(D,S)$-integral points $R\subset\{P\in C(\kbar)\mid [k(P):k]\leq d\}$.  Then the following statements hold.
\begin{enumerate}
\item There exists a morphism $\phi:C\to\mathbb{P}^1$, defined over $\kbar$, of degree $\leq d$ such that $\phi(\{P_1,\ldots, P_q\})\subset \{0,\infty\}$.
\item If $q>d$ or $C=\mathbb{P}^1$, then $R\setminus Z(D,d,k)$ is finite.\label{mii}
\end{enumerate}
\end{theorem}

Before proving Theorem \ref{mtheorem3} we give some preliminary results.  Let $C$ be a nonsingular projective curve over a number field $k$.  We will first relate integral points of bounded degree on $C$ to rational integral points on symmetric powers of $C$.  Let $d$ be a positive integer and let $\Sym^d(C)$ denote the $d$th symmetric power of $C$.  Let $\psi:C^d\to \Sym^d(C)$ be the natural map.  Let $\pi_i:C^d\to C$ be the $i$th projection map for $i=1,\ldots, d$.  Define 
\begin{align*}
C(k,d)=\{P\in C(\kbar)\mid [k(P):k]=d\}.
\end{align*}
To each point in $C(k,d)$ we can naturally associate a $k$-rational point of $\Sym^d(C)$ as follows.  Let $Q\in C(k,d)$.  Let $Q=Q_1,\ldots, Q_d$ denote the $d$ conjugates of $Q$ over $k$ (in some order).  Let $\rho(Q)=(Q_1,\ldots, Q_d)\in C^d(\kbar)$.  Since every element of the absolute Galois group $\Gal(\kbar/k)$ permutes the coordinates of $\rho(Q)$, it's clear that $\psi(\rho(Q))$ is a $k$-rational point of $\Sym^d(C)$.  The next lemma relates this construction and integrality.

\begin{lemma}
\label{intlem}
Let $D=\sum_{i=1}^q P_i$ be a sum of distinct points $P_i\in C(k)$ and let $S$ be a finite set of places of $k$ containing the archimedean places..  Let $R\subset C(k,d)$ be a set of $(D,S)$-integral points.  Let $E_i=\psi_*(P_i\times C^{d-1})=\psi_*\pi_1^*P_i$, $i=1,\ldots, q$, and let $E=\sum_{i=1}^qE_i$.  Then $\psi(\rho(R))$ is a set of $k$-rational $(E,S)$-integral points on $\Sym^d(C)$.
\end{lemma}

\begin{proof}
For all $Q\in R$, all $v\in M_k\setminus S$, and all embeddings of $\kbar$ in $\kbar_v$, we have, up to a function bounded by a constant $c_v$ as in Definition \ref{intdef},
\begin{align*}
\lambda_{E,v}(\psi(\rho(Q)))&=\sum_{i=1}^q \lambda_{E_i,v}(\psi(\rho(Q)))=\sum_{i=1}^q \lambda_{\psi_*\pi_1^*(P_i),v}(\psi(\rho(Q)))=\sum_{i=1}^q \lambda_{\psi^*\psi_*\pi_1^*(P_i),v}(\rho(Q))\\
&=\sum_{i=1}^q \lambda_{\sum_{j=1}^d \pi_j^*(P_i),v}(\rho(Q))=\sum_{i=1}^q \sum_{j=1}^d \lambda_{\pi_j^*(P_i),v}(\rho(Q))\\
&=\sum_{i=1}^q \sum_{j=1}^d \lambda_{P_i,v}(\pi_j(\rho(Q)))=\sum_{i=1}^q \sum_{j=1}^d \lambda_{P_i,v}(Q_j)\\
&=\sum_{j=1}^d \lambda_{D,v}(Q_j),
\end{align*}
where $Q_1,\ldots, Q_d$ are the $d$ conjugates of $Q$ over $k$.  Since $R$ is a set of $(D,S)$-integral points, it follows immediately from the definitions that $\psi(\rho(R))$ is a set of $(E,S)$-integral points.  The $k$-rationality of $\psi(\rho(R))$ was noted earlier.
\end{proof}

\begin{lemma}
\label{divlem}
The divisors $E_i$ of Lemma \ref{intlem} are all ample and algebraically equivalent.
\end{lemma}

\begin{proof}
Since $\psi$ is a finite morphism, each divisor $E_i$ is ample if and only if $\psi^*E_i$ is ample.  Then the ampleness of $E_i$ follows from the ampleness of $\psi^*E_i=\sum_{j=1}^d \pi_j^*P_i$.  Since any two points on a curve are algebraically equivalent and algebraic equivalence is preserved by flat pull-back and proper push-forward, it is immediate that the divisors $E_i$ are all algebraically equivalent.
\end{proof}

In order to effectively apply Lemma \ref{salem}, we need to exert some control over the term $\rk\{D_i|_W\}_{i=1}^q$ in the lemma.  This will be accomplished by the next lemma, which is a generalization of the classical fact that the intersection of $r$ ample effective divisors with an $r$-dimensional subvariety is nontrivial.  The proof uses an idea of Noguchi and Winkelmann from \cite[Lemma 7.3.2]{NW2}.  For convenience, we will frequently identify an effective divisor with its support.

\begin{lemma}
\label{ark}
Let $D_1,\ldots, D_q$ be ample effective Cartier divisors on a projective variety $X$.  Let $r=\rk\{D_i\}_{i=1}^q$ be the rank of the subgroup of $\Div(X)$ generated by $D_1,\ldots, D_q$.  Then for any subvariety $W\subset X$ with $\dim W\geq r$, $ \cap_{i=1}^q D_i\cap W\neq \emptyset$.  In particular, if $\rk\{D_i\}_{i=1}^q\leq \dim X$, then $\cap_{i=1}^q D_i\neq \emptyset$.
\end{lemma}

\begin{proof}
We first prove the contrapositive of the last statement in the theorem.  Let $n=\dim X$.  Suppose that $\cap_{i=1}^q D_i=\emptyset$.  Then we claim that $\rk\{D_i\}_{i=1}^q> n$.  Let $i_1=1$ and let $E_1$ be an irreducible component of $D_1=D_{i_1}$.  Note that $\dim E_1=n-1$.  Since $\cap_{i=1}^q D_i=\emptyset$, there exists a divisor $D_{i_2}\in \{D_1,\ldots, D_q\}$ such that $E_1\cap D_{i_2}\neq E_1$.  Let $E_2$ be an irreducible component of $E_1\cap D_{i_2}$.  Since $D_{i_2}$ is ample, $E_1\cap D_{i_2}$ has pure dimension $n-2$ and $\dim E_2=n-2$.  Continuing inductively, there exist distinct indices $i_1,\ldots, i_{n+1}$ and subvarieties $E_1,\ldots, E_n$ of $X$ such that $E_{m+1}\subset E_m$, $m=1,\ldots, n-1$, $E_m$ is an irreducible component of $\cap_{j=1}^m D_{i_j}$, $\dim E_m=n-m$, and $E_m\cap D_{i_{m+1}}\neq E_m$ for $m=1,\ldots, n$.  Then we claim that the divisors $D_{i_1},\ldots, D_{i_{n+1}}$ are independent in $\Div(X)$.  Suppose that for some $m<n+1$,
\begin{align}
\label{divr}
\sum_{j=m}^{n+1}a_jD_{i_j}=0
\end{align}
with $a_j\in\mathbb{Z}$ and $a_m\neq 0$.  Let $E$ be an irreducible component of $D_{i_m}$ containing $E_m$.  Since $E_m\not\subset D_{i_j}$, $m+1\leq j\leq n+1$, we have $E\not\subset D_{i_j}$, $m+1\leq j\leq n+1$.  Then considering the irreducible component $E$, clearly \eqref{divr} is impossible and we conclude that $\rk\{D_i\}_{i=1}^q> n$.

Now suppose that $\rk\{D_i\}_{i=1}^q=r$ and let $W\subset X$ be a subvariety with $\dim W\geq r$.  Reindex the divisors $D_i$ so that for some integer $q'$, $W\not\subset D_i$, $i=1,\ldots, q'$, and $W\subset D_i$, $i=q'+1,\ldots, q$.  Let $E_i=D_i|_W$, $i=1,\ldots, q'$.  Then $\rk\{E_i\}_{i=1}^{q'}\leq \rk\{D_i\}_{i=1}^q\leq r\leq \dim W$.  Since the divisors $E_i$ are ample, from the case proved above
\begin{align*}
\bigcap_{i=1}^{q'} E_i=\bigcap_{i=1}^q D_i\cap W\neq \emptyset.
\end{align*}

\end{proof}

We now prove Theorem \ref{mtheorem3}.

\begin{proof}[Proof of Theorem \ref{mtheorem3}]
Let $R\subset\{P\in C(\kbar)\mid [k(P):k]\leq d\}$ be an infinite set of $(D,S)$-integral points on $C$.  We may reduce to considering the case where every point of $R$ has exact degree $d$ over $k$, that is, $R\subset C(k,d)$.  Let $R'=\psi(\rho(R))$, where $\psi$ and $\rho$ are the maps introduced before Lemma~\ref{intlem}.  Then by Lemma \ref{intlem}, $R'$ is a set of $k$-rational $(E,S)$-integral points on $\Sym^d(C)$, where $E=\sum_{i=1}^qE_i$ and $E_i=\psi_*(P_i\times C^{d-1})$.  Since $\psi\circ \rho$ is at most $d$-to-$1$, $R'$ is an infinite set.  Let $W$ be a positive-dimensional irreducible component of the Zariski closure of $R'$.  Since $W\subset \Sym^d(C)$ contains a Zariski dense set of $k$-rational points, $W$ is necessarily defined over $k$.  Let $F_i=E_i|_W$, $i=1,\ldots, q$, and let $R''=R'\cap W$.

Fix a $k$-rational divisor $D_0$ of degree $d$ on $C$ (which always exists since $C(k,d)$ is infinite by assumption).  We will frequently identify points of $\Sym^d(C)$ with effective divisors of degree $d$ on $C$.  We will make use of the Abel-Jacobi map 
\begin{align*}
f:\Sym^d(C)&\to \Jac(C),\\
P_1+\cdots +P_d&\mapsto [P_1+\cdots +P_d-D_0],
\end{align*}
where the brackets denote the linear equivalence class of a divisor on $C$.  Let $X=f(W)$.  The proof will be divided into the two cases $\dim X=0$ and $\dim X>0$.

First suppose that $X$ is a point.  This implies that $W$ is contained in a complete linear system $V=|D'|\subset \Sym^d(C)$ for some effective divisor $D'$ of degree $d$, where $f(X)=\{[D'-D_0]\}$.  Let $H_i=E_i|_V$, $i=1,\ldots, q$, and let $\mathcal{H}=\{H_1,\ldots, H_q\}$.  Then $R''$ is a set of $(\sum_{i=1}^q H_i,S)$-integral points on $V$.  Let $\dim |D'|=r$.  Since $\dim W>0$, we have $r>0$.  We may identify $V=|D'|$ with $r$-dimensional projective space $\mathbb{P}^r$.  Then each divisor $H_i$ is a hyperplane in $V$ and corresponds to the linear system $P_i+|D'-P_i|$.

By repeated application of Lemma \ref{hyplem}, it follows that there exists a finite union $Z$ of positive-dimensional linear subspaces $L\subset V$ over $k$ such that 
\renewcommand{\theenumi}{\arabic{enumi}}
\begin{enumerate}
\item $R''\setminus Z$ is a finite set.
\item For each linear subspace $L$ in the finite union, $\mathcal{H}|_L$ consists of $\leq \dim L+1$ hyperplanes of $L$ in general position.
\end{enumerate}
Let $L\subset V$ be an $n$-dimensional linear space over $k$, $n>0$, such that $\mathcal{H}|_L$ consists of $m\leq n+1$ hyperplanes $H_1',\ldots, H_m'$ of $L$ in general position.  Let
\begin{align*}
I_1&=\{i\in \{1,\ldots, q\}\mid H_i|_L=H_{j}', j\in \{1,\ldots, m-1\}\},\\
I_2&=\{i\in \{1,\ldots, q\}\mid H_i|_L=H_{m}'\}.
\end{align*}

Since the hyperplanes $H_1',\ldots, H_m'$ are defined over $k$, in general position, and $m-1\leq n$, the intersection $\cap_{i=1}^{m-1} H_i'=\cap_{i\in I_1}H_i\cap L$ is nonempty and contains a $k$-rational point $P_1'$.  Let $P'\in R''\cap L$.  Let $L'$ be the line through $P'$ and $P_1'$ and let $\{P_2'\}=L'\cap H_m'$.  If the points $P_1',P_2'\in \Sym^d(C)(k)$ correspond to the effective $k$-rational divisors $D_1'$ and $D_2'$, respectively, then $L'$ corresponds to a one-dimensional linear system associated to a map $\phi:C\to\mathbb{P}^1$ with $\dv(\phi)=D'_1-D'_2$.  From the definitions, $\sum_{i\in I_1}P_i\leq D_1'$ and $\sum_{i\in I_2}P_i\leq D_2'$.  Since $L'$ isn't contained in any of the hyperplanes $H_i'$, no $P_i$ is contained in the support of both $D_1'$ and $D_2'$.  Then $\phi$ has the property that $\deg\phi\leq d$ and $\phi(\{P_1,\ldots, P_q\})\subset \{0,\infty\}$.  Moreover, since $P'\in L'(k)$, it follows that $(\psi\circ \rho)^{-1}(P')\subset\phi^{-1}(k)$.  Therefore, $(\psi\circ \rho)^{-1}(R'')\setminus Z(D,d,k)$ is a finite set.

Before considering the case $\dim X>0$, we prove part \eqref{mii}.  If $C=\mathbb{P}^1$, then $\Jac(C)$ is a point and $X=f(W)$ is a point for any $W\subset \Sym^d (C)$.  It then follows from the above case that $R\setminus Z(D,d,k)$ is a finite set.

Suppose now that $q>d$.  Then we will show that $f(W)$ is a point.  If $q>d$, then $\cap_{i=1}^q \Supp E_i=\emptyset$.  It follows that $\cap_{i=1}^q\Supp F_i=\emptyset$.  We now show that if $\cap_{i=1}^q \Supp F_i=\emptyset$, then $f(W)$ is a point.  Let $r=\rk\{F_i\}_{i=1}^q$ be the rank of the subgroup of $\Div(W)$ generated by the divisors $F_1,\ldots, F_q$.  By Lemma \ref{divlem}, all of the divisors $E_i$ are algebraically equivalent and so $\rk \{c_1(E_i)\}_{i=1}^q=1$.  Then since $R''$ is a Zariski dense set of $k$-rational $(\sum_{i=1}^q F_i,S)$-integral points on $W$, by Lemma~\ref{salem}, $r+q(W)\leq \dim W+1$.  Therefore,
\begin{align}
\label{qeq}
q(W)\leq \dim W+1-r.
\end{align}
Now since $\cap_{i=1}^q \Supp F_i=\emptyset$ and each divisor $F_i$ is ample, by Lemma \ref{ark} we have $r\geq \dim W+1$.  Therefore, by \eqref{qeq} we must have $q(W)=0$.  It follows that $W$ does not admit a nonconstant map to an abelian variety and $f(W)$ consists of a single point in $\Jac(C)$.  Therefore, from the case proved above, if $q>d$ then $R\setminus Z(D,d,k)$ is a finite set.

Suppose now that $\dim X>0$.  Then from the previous argument, we must have that $\cap_{i=1}^q \Supp F_i$ is nonempty. Let $Y=f(\cap_{i=1}^q \Supp F_i)$.  Suppose first that for some $P\in Y(\kbar)$, the fiber $f^{-1}(P)$ is not contained in the support of $E$.  Let $D'\in \cap_{i=1}^q (\Supp F_i)(\kbar)$ with $f(D')=P$.  Then, viewing $D'$ as an effective divisor, we have $\sum_{i=1}^q P_i\leq D'$ and $\deg D'=d$.  Moreover, since $f^{-1}(P)$ is not contained in the support of $E$, none of the $P_i$ are basepoints of the complete linear system $|D'|$.  It follows that there exists a morphism, over $\kbar$, $\phi:C\to\mathbb{P}^1$ with $\phi(\{P_1,\ldots, P_q\})=\{0\}$ and $\deg \phi\leq d$, proving the theorem in this case.  Therefore, we may assume now that $f^{-1}(Y)$ is completely contained in the support of $E$.  

Since the set of $k$-rational points $f(R'')$ is Zariski dense in $X=f(W)\subset \Jac(C)$, by Faltings' Theorem \ref{Falt1}, $X$ must be a translate of an abelian subvariety of $\Jac(C)$.  Suppose that $Y$ has an irreducible component (over $\kbar$) $E'$ of codimension one in $X$.  Without loss of generality, after possibly replacing $k$ by a finite extension, we may assume that $E'$ is defined over $k$.  Since $f^{-1}(Y)\subset \Supp E$, it follows that $f(R'')$ is a set of $(E',S)$-integral points of $X$.  Then by Theorem \ref{Voj2}, $f(R'')$ is not Zariski dense in $X$.  This contradicts our assumption that $R''$ is Zariski dense in $W$.  So we may assume now that every component of $Y$ has codimension at least two in $X$.  We now show that this is impossible.  By \eqref{qeq} and the universal property of the Albanese variety, we have 
\begin{align*}
\dim X=\dim f(W)\leq q(W)\leq \dim W+1-r,
\end{align*}
where $r=\rk\{F_i\}_{i=1}^q$ is the rank of the subgroup of $\Div(W)$ generated by the divisors $F_1,\ldots, F_q$.  Then we have the inequality of relative dimensions $\dim W-\dim X\geq r-1$.  Since every component of $Y$ has codimension at least two in $X$, there exists a curve $C\subset X$ not intersecting $Y$.  Let $C'$ be an irreducible component of $f^{-1}(C)$ with $f(C')=C$.  Since $\dim W-\dim X\geq r-1$ and $\dim C=1$, we have $\dim C'\geq r$.  Now since $\rk\{F_i\}_{i=1}^q=r$, by Lemma \ref{ark}, $\cap_{i=1}^q(\Supp F_i)\cap C'\neq \emptyset$.  This contradicts the fact that $Y\cap C=\emptyset$.

\end{proof}

\section{Some examples and open questions}
\label{sopen}

First, we give an example showing that if $C$ has infinitely many integral points of degree $d$ over a number field $k$, then there may not exist any morphism $\phi$ as in Theorem \ref{mtheorem2} (or Theorem \ref{mtheorem}) that is definable over $k$, even allowing $\mathbb{G}_m$ to be replaced by a twist of $\mathbb{G}_m$ (or $\{0,\infty\}$ by a pair of points conjugate over $k$).

\begin{example}
\label{ex1}
Let $k$ be a number field and let $S$ be a finite set of places of $k$ containing the archimedean places.  Let $P_1\in \mathbb{P}^1(\kbar)$ with $[k(P_1):k]=3$ and let $P_2$ and $P_3$ be the two other points conjugate to $P_1$ over $k$.  Let $D=P_1+P_2+P_3$, a divisor on $\mathbb{P}^1$ defined over $k$.  Then we claim that there exists an infinite set $R$ of $(D,S)$-integral points that are quadratic over $k$.  To accomplish this we will take advantage of the construction in \cite[pp. 99-103]{Lev2}.  Let $P_i=(\alpha_i,1)$, $i=1,2,3$.  Let
\begin{align*}
R'=\{(a,b,c)\in k^3\mid a\alpha_1^2+b\alpha_1+c\in \O_{k(\alpha_1)}^*,a\neq 0\}.
\end{align*}
Let
\begin{align*}
R=\{x\in \kbar\mid ax^2+bx+c=0 \text{ for some }(a,b,c)\in R'\}.
\end{align*}
From \cite[Th. 7]{Lev2}, if we view $R'$ as a subset of $\mathbb{P}^2$ then it is Zariski dense in $\mathbb{P}^2$.  In particular, $R$ is an infinite set.  Furthermore, $R$ is a set of $(D,S)$-integral points.  To show this, it suffices to show that for some nonzero constant $C$, $\frac{C}{x-\alpha_i}\in \O_{k(x,\alpha_i)}$ for $i=1,2,3$ and all $x\in R$.  First, there exists a positive integer $m$ such that for all $(a,b,c)\in R'$, $ma,mb,mc\in \O_k$.  Let $n$ be a positive integer such that $n\alpha_i\in \O_{k(\alpha_i)}$ for $i=1,2,3$.  Let $C=mn$.  Let $x\in R$ be a root of $az^2+bz+c$ with $(a,b,c)\in R'$.  Let $x'$ be the other root of $az^2+bz+c$.  Then
\begin{align*}
\frac{C}{x-\alpha_i}=\frac{amn(x'-\alpha_i)}{a(x-\alpha_i)(x'-\alpha_i)}=\frac{amn(x'-\alpha_i)}{a\alpha_i^2+b\alpha_i+c}.
\end{align*}
From our definitions, $a\alpha_i^2+b\alpha_i+c\in \O_{k(\alpha_i)}^*$.  We also have $amn\alpha_i\in \O_{k(\alpha_i)}$.  Since $amx'$ satisfies $(amx')^2+mb(amx')+(ma)(mc)=0$, we have $amx'\in \O_{k(x')}=\O_{k(x)}$.  Thus, $amn(x'-\alpha_i)\in \O_{k(x,\alpha_i)}$.  So $\frac{C}{x-\alpha_i}\in \O_{k(x,\alpha_i)}$ for $i=1,2,3$ and all $x\in R$.    

So there exists an infinite set $R$ of $(D,S)$-integral points that are quadratic over $k$.  On the other hand, since $P_1$, $P_2$, and $P_3$ are conjugate over $k$, it is immediate that there is no quadratic map $\phi:\mathbb{P}^1\to \mathbb{P}^1$ {\em over $k$} such that $\phi(\{P_1,P_2,P_3\})$ consists of only one or two points.
\end{example}

In the other direction, we note that if there exists a morphism $\phi$, defined over $k$, as in Theorem \ref{mtheorem}, then it may be necessary to enlarge $S$ (or extend $k$) so that there are infinitely many $S$-integral points of degree $d$ or less over $k$ on the affine curve.

\begin{example}
\label{ex2}
Let $d$ be a positive integer.  Let $C$ be a nonsingular projective curve over $\mathbb{Q}$ and let $P_1,\ldots, P_q\in C(\mathbb{Q})$ be distinct points with $q>d$.  Let $S=\{\infty\}$ and let $D=\sum_{i=1}^qP_i$.  Then it is a consequence of Bombieri's version of Runge's method \cite{Bomb} that all sets $R\subset\{P\in C(\Qbar)\mid [\mathbb{Q}(P):\mathbb{Q}]\leq d\}$ of $(D,S)$-integral points of degree $d$ or less over $\mathbb{Q}$ are finite.  Note that if $C=\mathbb{P}^1$ and $q\leq 2d$, then there exist morphisms $\phi:C\to\mathbb{P}^1$ over $\mathbb{Q}$ with $\deg \phi=d$ and $\phi(\{P_1,\ldots, P_q\})\subset\{0,\infty\}$.
\end{example}

In view of Example \ref{ex1}, we raise the following question.

\begin{question}
If the set $\{P\in C(\overline{\O}_{L,S})\mid [L(P):L]\leq d\}$ is infinite, can the morphism $\phi$ of Theorem \ref{mtheorem} always be taken to be defined over $L(P_1,\ldots, P_q)$?
\end{question}

Note that Theorem \ref{mtheorem3}\eqref{mii}, where it is assumed that $P_1,\ldots, P_q$ are $k$-rational, answers this question positively when $q>d$ or $C=\mathbb{P}^1$.  Going further, we may ask whether Theorem \ref{mtheorem3}\eqref{mii} holds in general:

\begin{question}
\label{mainq}
Let $C$ be a nonsingular projective curve defined over a number field $k$.  Let $S$ be a finite set of places of $k$ containing the archimedean places.  Let $D=\sum_{i=1}^q P_i$ be a nontrivial divisor defined over $k$, where $P_i\in C(\kbar)$, $i=1,\ldots, q$, are distinct points.  Let $d$ be a positive integer.  Let $R\subset\{P\in C(\kbar)\mid [k(P):k]\leq d\}$ be a set of $(D,S)$-integral points.  Is the set $R\setminus Z(D,d,k(P_1,\ldots, P_q))$ always finite?
\end{question}

Along the lines of Question \ref{mainq}, we prove the following result.

\begin{theorem}
Under the same hypotheses as Question \ref{mainq}, let
\begin{align*}
Z'(C,d,k)=\bigcup_{\substack{\phi\in k(C)\\\deg \phi\leq d}} \phi^{-1}(k).
\end{align*}
Then the set
\begin{align*}
R\setminus Z'(C,d,k)
\end{align*}
is finite.
\end{theorem}

\begin{proof}
We may reduce to considering a set of $(D,S)$-integral points $R\subset C(k,d)$.  Let $R'=\psi(\rho(R))$.  Let $X$ be the subset of $\Sym^d (C)$ defined by
\begin{align*}
X=\{E\in \Sym^d(C)(\kbar)\mid \dim |E|>0\}.
\end{align*}
Let $g$ be the genus of $C$.  If $d>g$, then $X=\Sym^d (C)$ by Riemann-Roch.  If $d\leq g$, then $X$ is a proper closed subset of $\Sym^d (C)$.  Let $f:\Sym^d (C)\to \Jac(C)$ and $E$ be as in the proof of Theorem \ref{mtheorem3}.  So $R'$ is a set of $k$-rational $(E,S)$-integral points on $\Sym^d (C)$.  It suffices to show that $R'\setminus X$ is a finite set.  Indeed, if $P_1+\cdots +P_d\in X(k)$, then by the definition of $X$, there exists a rational function $\phi\in k(C)$ with $\deg \phi\leq d$ and $\phi^{-1}(0)\subset\{P_1,\ldots, P_d\}$.  If the points $P_i$ are all conjugate over $k$, then we must have $\phi^{-1}(0)=\{P_1,\ldots, P_d\}\subset Z'(C,d,k)$.  It then follows from the definitions that $R\setminus Z'(C,d,k)$ is finite.

To complete the proof, we now show that $R'\setminus X$ is a finite set.  If $f(\Supp E)=\Jac(C)$, then $d> g$, $X=\Sym^d (C)$, and there is nothing to prove.  Otherwise, let $E'=f(\Supp E)$ with, say, its reduced induced subscheme structure.  By a result of Faltings \cite{Fal1}, for any $\epsilon>0$, 
\begin{align*}
\sum_{v\in S}\lambda_{E',v}(P)<\epsilon h_A(P)
\end{align*}
for all but finitely many points $P\in \Jac(C)(k)\setminus E'$, where $h_A$ is a height associated to any ample divisor $A$ on $\Jac(C)$ (see also \cite{Sil} for the theory of heights associated to closed subschemes).  Note that for any finite set $T\subset \Jac(C)(\kbar)$, the set $f^{-1}(T\cup E')\setminus (X\cup \Supp E)$ consists of only finitely many points.  Then since $E\subset f^* E'$, by functoriality we obtain
\begin{align*}
\sum_{v\in S}\lambda_{E,v}(P)&\leq \sum_{v\in S}\lambda_{f^*E',v}(P)+O(1)=\sum_{v\in S}\lambda_{E',v}(f(P))+O(1)\\
&<\epsilon h_A(f(P))+O(1)\\
&<\epsilon h_{A'}(P)+O(1),
\end{align*}
for all but finitely many points $P\in \Sym^d (C)(k)\setminus (X\cup \Supp E)$, where we let $A'=f^*A$.  Then for all $P\in R'\setminus X$,
\begin{align*}
\sum_{v\in S}\lambda_{E,v}(P)=h_E(P)+O(1)<\epsilon h_{A'}(P)+O(1).
\end{align*}
Since $E$ is an ample divisor, this implies that $R'\setminus X$ is finite.
\end{proof}

\subsection*{Acknowledgments}
The author would like to thank Junjiro Noguchi for helpful comments on an earlier draft.

\bibliography{Bounded}
\end{document}